\documentclass[12pt]{amsart}

\usepackage{amssymb}
\usepackage{amscd}

\title[Algebraic-geometric codes from Galois points]{Algebraic-geometric codes with many automorphisms arising from Galois points}

\author{Satoru Fukasawa}

\address{Faculty of Science, Yamagata University, Kojirakawa-machi 1-4-12, Yamagata 990-8560, Japan} 
\email{s.fukasawa@sci.kj.yamagata-u.ac.jp}

\thanks{The author was partially supported by JSPS KAKENHI Grant Numbers JP19K03438 and JP22K03223.}

\newtheorem{theorem}{Theorem}

\newtheorem{example}{Example}

\theoremstyle{definition}
\newtheorem{remark}{Remark}

\begin{document}
\begin{abstract} 
A method of constructing algebraic-geometric codes with many automorphisms arising from Galois points for algebraic curves is presented.   
\end{abstract}

\maketitle 

\section{Introduction} 
A finite field of $q$ elements is denoted by $\mathbb{F}_q$. 
The set of all $\mathbb{F}_q$-rational points of an algebraic variety $V$ over $\mathbb{F}_q$ is denoted by $V(\mathbb{F}_q)$. 
V. D. Goppa discovered a construction of algebraic-geometric codes, according to the following data: a smooth projective curve $X$ defined over $\mathbb{F}_q$, a subset $S \subset X(\mathbb{F}_q)$,  and a divisor $D$ over $\mathbb{F}_q$ (see, for example, \cite{stichtenoth2, tsfasman-vladut}). 
This paper presents a method of constructing algebraic-geometric codes with many automorphisms, using the theory of {\it Galois points}.  

Let $\varphi: X \rightarrow \mathbb{P}^2$ be a morphism defined over $\mathbb{F}_q$, which is birational onto $\varphi(X)$. 
A point $P \in \mathbb{P}^2 \setminus \varphi(X)$ is called an outer Galois point, if the field extension $\overline{\mathbb{F}}_q(X)/(\pi_P\circ\varphi)^*\overline{\mathbb{F}}_q(\mathbb{P}^1)$ of function fields induced by the projection $\pi_P: \varphi(X) \rightarrow \mathbb{P}^1$ from $P$ is a Galois extension (see \cite{fukasawa1, miura-yoshihara, yoshihara}, see also \cite[Remark 1]{fukasawa3} for the definition of a Galois point over $\mathbb{F}_q$). 
The main result is the following. 

\begin{theorem} \label{code}
Let $G_1, G_2 \subset {\rm Aut}_{\mathbb{F}_q}(X)$ be different finite subgroups with $|G_1|=|G_2|$, and let $Q, Q' \in X(\overline{\mathbb{F}}_q)$. 
Assume that the following four conditions are satisfied: 
\begin{itemize}
\item[(a)] $X/G_i \cong \mathbb{P}^1$ for $i=1, 2$;  
\item[(b)] $G_1 \cap G_2=\{1\}$; 
\item[(c)] $\sum_{\sigma \in G_1}\sigma(Q)=\sum_{\tau \in G_2}\tau(Q)$ over $\overline{\mathbb{F}}_q$, and this divisor $D$ is defined over $\mathbb{F}_q$; 
\item[(d)] $Q' \in X(\mathbb{F}_q)$, $Q' \not\in {\rm supp}(D)$, and the cardinality $\#S$ of the orbit $S:=\langle G_1, G_2 \rangle \cdot Q'$ is at least $|G_1|+1$. 
\end{itemize} 
Then the following hold.  
\begin{itemize}
\item[(1)] The evaluation map 
$$ \Phi: \mathcal{L}(D) \rightarrow \bigoplus_{R \in S} \mathbb{F}_q \cdot R; \ f \mapsto (f(R))_{R \in S}$$
is injective, and the minimum distance of the code $\Phi(\mathcal{L}(D))$ is at least $\#S-|G_1|$. 
\item[(2)] There exists an injective homomorphism 
$$ \langle G_1, G_2 \rangle \hookrightarrow {\rm Aut}(\Phi(\mathcal{L}(D))).  $$ 
\end{itemize} 
\end{theorem} 

\begin{remark}
It follows from a criterion described by the present author in \cite{fukasawa2} that conditions (a), (b) and (c) are satisfied if and only if there exists a morphism $\varphi: X \rightarrow \mathbb{P}^2$ of degree $|G_1|$ defined over $\mathbb{F}_q$, which is birational onto its image, and different outer Galois points $P_1, P_2 \in (\mathbb{P}^2\setminus \varphi(X))(\mathbb{F}_q)$ exist for $\varphi(X)$ such that $G_{P_1}=G_1$, $G_{P_2}=G_2$, and points $P_1, P_2, Q$ are collinear, where $G_{P_i} \subset {\rm Aut}_{\mathbb{F}_q}(X)$ is the Galois group at $P_i$ for $i=1, 2$. 
\end{remark}

\section{Proof} 

\begin{proof}[Proof of Theorem \ref{code}] 
Let $f \in \mathcal{L}(D) \setminus \{0\}$. 
Since $\deg D=|G_1|$, it follows that the number of zeros of $f$ is at most $|G_1|$. 
By condition (d), $\#S \ge |G_1|+1$.  Therefore, $f(R) \ne 0$ for some $R \in S$. 
This argument also implies that the minimum distance is at least $\#S-|G_1|$.  
Assertion (1) follows. 

We prove assertion (2). 
By condition (c), it follows that $\gamma^*D=D$, namely, $\mathcal{L}(\gamma^*D)=\mathcal{L}(D)$, for any $\gamma \in \langle G_1, G_2 \rangle$. 
Furthermore, $\gamma(S)=S$ for any $\gamma \in \langle G_1, G_2 \rangle$, since $S$ is an orbit. 
Therefore, we have a natural homomorphism 
$$ \langle G_1, G_2 \rangle \rightarrow {\rm Aut}(\Phi(\mathcal{L}(D)) $$
(see also \cite{stichtenoth1}, \cite[8.2.3]{stichtenoth2}).   
We prove that there does not exist $\gamma \in \langle G_1, G_2 \rangle \setminus \{1\}$ such that $\gamma|_S=1$.  
Assume that $\gamma \in \langle G_1, G_2 \rangle$ and $\gamma|_S=1$. 
Let $D_i:=\sum_{\sigma \in G_i}\sigma(Q')$ for $i=1, 2$. 
By conditions (a), (b) and (c), it follows from the criterion by the present author in \cite{fukasawa2} that the coverings $X \rightarrow X/G_1$, $X \rightarrow X/G_2$ are realized as projections from different outer Galois points $P_1, P_2$ with $G_{P_1}=G_1$, $G_{P_2}=G_2$. 
The birational embedding $X \rightarrow \mathbb{P}^2$ is denoted by $\varphi$.   
Then $D_1, D_2 \in \Lambda$ for $i=1, 2$, where $\Lambda$ is the sublinear system of $|D|$ corresponding to $\varphi$. 
Note that $D_1 \ne D_2$, $D_1 \ne D$, and $D_2 \ne D$.  
Since $\gamma|_S=1$, it follows that $\gamma^*D_1=D_1$ and $\gamma^*D_2=D_2$. 
Since $\gamma^*D=D$, it follows that $\gamma^*\Lambda=\Lambda$, namely, there exists a linear transformation $\tilde{\gamma}$ of $\mathbb{P}^2$ such that $\tilde{\gamma} \circ \varphi=\varphi \circ \gamma$. 
Let $\ell, \ell_1$ and $\ell_2 \subset \mathbb{P}^2$ be the lines corresponding to $D$, $D_1$ and $D_2$, respectively. 
Then the linear transformation $\tilde{\gamma}$ fixes non-collinear points $P_1, P_2$ and $\varphi(Q') \in \mathbb{P}^2$, which are given by $\ell \cap \ell_1$, $\ell \cap \ell_2$ and $\ell_1 \cap \ell_2$, respectively.  
If there exists a point $R \in \varphi(S)$ with $R \not\in \ell_1 \cup \ell_2$, then $\tilde{\gamma}=1$, namely, $\gamma=1$. 
Therefore, we can assume that $\varphi(S) \subset \ell_1 \cup \ell_2$. 
By condition (d), $\varphi(S) \not\subset \ell_1$. 
Therefore, there exists a point $R_2 \in \varphi(S) \cap \ell_2$ with $R_2 \not\in \ell_1$. 
In particular, $R_2 \ne \varphi(Q')$. 
Similarly, there exists a point $R_1 \in \varphi(S) \cap \ell_1$ with $R_1 \ne \varphi(Q')$. 
Since $\tilde{\gamma}$ fixes three points $P_1, \varphi(Q'), R_1$ on the line $\ell_1$, it follows that $\tilde{\gamma}|_{\ell_1}=1$. 
Similarly, $\tilde{\gamma}|_{\ell_2}=1$. 
Therefore, $\tilde{\gamma}=1$, namely, $\gamma=1$. 
\end{proof} 

\begin{remark} \label{many Galois} 
Theorem \ref{code} can be generalized to the case where collinear outer Galois points $P_1, P_2, \ldots, P_k$ exist, namely, there exists an $\mathbb{F}_q$-line $\ell \subset \mathbb{P}^2$ such that $P_1, P_2, \ldots, P_k \in \ell$. 
Assume the following:  
\begin{itemize}
\item[(d')] $Q' \in X(\mathbb{F}_q)$, $Q' \not\in {\rm supp}(\varphi^*\ell)$, and there exist $i, j$ with $i \ne j$ such that $\#(\langle G_{P_i}, G_{P_j}\rangle \cdot Q') \ge |G_{P_i}|+1$. 
\end{itemize} 
Let $S:=\langle G_{P_1}, G_{P_2}, \ldots, G_{P_k} \rangle \cdot Q'$. 
Then there exists an injective homomorphism 
$$ \langle G_{P_1}, G_{P_2}, \ldots, G_{P_k} \rangle \hookrightarrow {\rm Aut}(\Phi(\mathcal{L}(\varphi^*\ell))).  $$ 
\end{remark} 

\section{Examples} 

\begin{example} \label{Hermitian} 
Let $X \subset \mathbb{P}^2$ be the Fermat curve 
$$ X^{q+1}+Y^{q+1}+Z^{q+1}=0.  $$
It is confirmed that $P_1=(1:0:0)$ and $P_2=(0:1:0)$ are outer Galois points with $G_{P_i} \cong \mathbb{Z}/(q+1)\mathbb{Z}$ for $i=1, 2$ (see also \cite{homma}). 
Conditions (a) and (b) in Theorem \ref{code} are satisfied for $G_{P_1}, G_{P_2}$. 
Let $\zeta$ be a primitive $q+1$-th root of unity. 
Then 
$$ \langle G_{P_1}, G_{P_2} \rangle
=\{(X:Y:Z) \mapsto (\zeta^i X: \zeta ^j Y: Z) \ | \ 0 \le i, j \le q \}. $$ 
Let $Q \in X(\mathbb{F}_{q^2}) \cap \{Z=0\}$. 
Then condition (c) in Theorem \ref{code} is satisfied, and $D$ is the divisor defined by $X \cap \{Z=0\}$.  
We take a point $Q'=(a:b:1) \in X(\mathbb{F}_{q^2})$ with $a b \ne 0$. 
Then 
$$ S= \{ (\zeta^i a: \zeta ^j b: 1) \ | \ 0 \le i, j \le q \}.$$ 
Therefore, condition (d) in Theorem \ref{code} is satisfied for the $4$-tuple $(G_{P_1}, G_{P_2}, Q, Q')$. 
It is well known that $\dim \mathcal{L}(D)=3$ (see, for example, \cite[Lemma 11.28]{hkt}). 
Note that $q+1$ points $(\zeta^ia: b:1) \in S$ with $0 \le i \le q$ are zeros of the function $y-b \in \mathcal{L}(D)$. 
Then the code $\Phi(\mathcal{L}(D))$ has parameters 
$$[(q+1)^2, \ 3, \ (q+1)q]_{q^2}$$ 
and admits an automorphism group 
$$(\mathbb{Z}/(q+1)\mathbb{Z})^{\oplus 2}.$$ 
\end{example}

\begin{example} 
Let $X=\mathbb{P}^1$. 
Assume that $q$ is odd and $q \ge 5$. 
Let $\zeta$ be a primitive $m:=(q-1)/2$-th root of unity, and let 
\begin{eqnarray*} 
G_1&:=&\{(s:t) \mapsto (\zeta^i s:t) \ | \ 0 \le i \le m-1\},  \\
G_2&:=&\{(s:t) \mapsto (\zeta^i s+(1-\zeta^i)t:t) \ | \ 0 \le i \le m-1\}. 
\end{eqnarray*} 
By L\"{u}roth's theorem, condition (a) in Theorem \ref{code} is satisfied. 
Obviously, $G_1 \cap G_2=\{1\}$, namely, condition (b) is satisfied. 
We take $Q=(1:0)$.  
Then condition (c) in Theorem \ref{code} is satisfied, and $D=m Q$. 
Let $Q'=(0:1)$. 
It can be confirmed that $$ \langle G_1, G_2 \rangle \cong \mathbb{F}_q \rtimes (\mathbb{Z}/m\mathbb{Z}), $$
and that $\langle G_1, G_2 \rangle \cdot Q'=\mathbb{P}^1(\mathbb{F}_q) \setminus \{Q\}$. 
Therefore, condition (d) in Theorem \ref{code} is satisfied for the $4$-tuple $(G_1, G_2, Q, Q')$.  
It is well known that $\dim \mathcal{L}(D)=m+1$ (see, for example, \cite{stichtenoth2}). 
Then the (Reed--Solomon) code $\Phi(\mathcal{L}(D))$ has parameters
$$ \left[q, \ \frac{q+1}{2}, \ \frac{q+1}{2} \right]_q $$
and admits an automorphism group of order $q(q-1)/2$. 
\end{example}

\begin{example}
Let $\varphi(X) \subset \mathbb{P}^2$ be defined by 
$$ (x^{q^2}+x)^{q+1}+(y^{q^2}+y)^{q+1}+1=0, $$
which was studied in \cite{borges-fukasawa}. 
According to \cite{borges-fukasawa}, the full automorphism group $G:={\rm Aut}(X)$ is generated by the Galois groups of all outer Galois points on the line $\{Z=0\}$, and $|G|=q^5(q^2-1)(q+1)$. 
The divisor coming from $\varphi(X) \cap \{Z=0\}$ is denoted by $D$. 
Let $a^{q+1}=-1$, let $x_0^{q^2}+x_0=a$ and let $Q' \in X(\mathbb{F}_{q^4})$ with $\varphi(Q')=(x_0:0:1)$. 
It is easily verified that $\#(\langle G_{P_1}, G_{P_2} \rangle \cdot Q')>\deg \varphi(X)=|G_{P_1}|$ for Galois points $P_1=(1:0:0)$ and $P_2=(0:1:0)$. 
This implies that condition (d') in Remark \ref{many Galois} is satisfied. 
It can be confirmed that $\#(G \cdot Q')=q^5(q^2-1)$. 
It was proved that $\dim \mathcal{L}(D)=3$ in \cite[Proposition 1]{borges-fukasawa}. 
Note that $q^3+q^2$ points $(x:0:1) \in G \cdot Q'$ with $(x^{q^2}+x)^{q+1}+1=0$ are zeros of the function $y \in \mathcal{L}(D)$. 
Then the code $\Phi(\mathcal{L}(D))$ has parameters 
$$ [q^7-q^5, \ 3, \ q^7-q^5-q^3-q^2]_{q^4} $$
and admits an automorphism group of order $q^5(q^2-1)(q+1)$. 
\end{example}

\section{A generalization} 
T. Hasegawa \cite{hasegawa} suggested a generalization of our code in Example \ref{Hermitian}, according to replacing the divisor $D$ by $m D$ with $m <\#S/|G_1|$. 
Based on this, Theorem \ref{code} was generalized by the author as follows. 

\begin{theorem}
Let $(G_1, G_2, Q, Q')$ be a $4$-tuple satisfying conditions (a), (b) and (c) in Theorem \ref{code}. 
We take a positive integer $e$ so that $D=e D_0$, where $D_0=\sum_{R \in {\rm supp}(D)} R$. 
Let $m$ be a positive integer. 
Assume the following: 
\begin{itemize}
\item[(d'')] $Q' \in X(\mathbb{F}_q)$, $Q' \not\in {\rm supp}(D)$, and  $m < e \times \#S/|G_1|$, where $S:=\langle G_1, G_2 \rangle \cdot Q'$. 
\end{itemize} 
Then the following hold.  
\begin{itemize}
\item[(1)] The evaluation map 
$$ \Phi: \mathcal{L}(m D_0) \rightarrow \bigoplus_{R \in S} \mathbb{F}_q \cdot R; \ f \mapsto (f(R))_{R \in S}$$
is injective, and the minimum distance of the code $\Phi(\mathcal{L}(m D_0))$ is at least $\#S-m \times |G_1|/e$. 
\item[(2)] There exists an injective homomorphism 
$$ \langle G_1, G_2 \rangle \hookrightarrow {\rm Aut}(\Phi(\mathcal{L}(m D_0))).  $$ 
\end{itemize} 
\end{theorem} 

\

\begin{center} {\bf Acknowledgements} \end{center} 
The author is grateful to Professor Masaaki Homma, Professor Takehiro Hasegawa, and Doctor Kazuki Higashine for helpful comments.


\begin{thebibliography}{20} 
\bibitem{borges-fukasawa} H. Borges and S. Fukasawa, An elementary abelian $p$-cover of the Hermitian curve with many automorphisms, Math. Z. {\bf 302} (2022), 695--706. 
\bibitem{fukasawa1} S. Fukasawa, Galois points for a plane curve in arbitrary characteristic, Geom. Dedicata {\bf 139} (2009), 211--218.  
\bibitem{fukasawa2} S. Fukasawa, A birational embedding of an algebraic curve into a projective plane with two Galois points, J. Algebra {\bf 511} (2018), 95--101. 
\bibitem{fukasawa3} S. Fukasawa, Galois points and rational functions with small value sets, Hiroshima Math. J., to appear. 
\bibitem{hasegawa} T. Hasegawa, Private communications, November, 2022. 
\bibitem{hkt} J. W. P. Hirschfeld, G. Korchm\'{a}ros and F. Torres, Algebraic Curves over a Finite Field, Princeton Univ. Press, Princeton, 2008. 
\bibitem{homma} M. Homma, Galois points for a Hermitian curve, Comm. Algebra {\bf 34} (2006), 4503--4511. 
\bibitem{miura-yoshihara} K. Miura and H. Yoshihara, Field theory for function fields of plane quartic curves, J. Algebra {\bf 226} (2000), 283--294. 
\bibitem{stichtenoth1} H. Stichtenoth, On automorphisms of geometric Goppa codes, J. Algebra {\bf 130} (1990), 113--121. 
\bibitem{stichtenoth2} H. Stichtenoth, Algebraic Function Fields and Codes, Graduate Texts in Mathematics {\bf 254}, Springer-Verlag, Berlin Heidelberg, 2009.  
\bibitem{tsfasman-vladut} M. A. Tsfasman and S. G. Vl\v{a}du\c{t}, Algebraic-Geometric Codes, Mathematics and its Applications, {\bf 58}, Kluwer Academic Publishers, Dordrecht, 1991. 
\bibitem{yoshihara} H. Yoshihara, Function field theory of plane curves by dual curves, J. Algebra {\bf 239} (2001), 340--355. 
\end{thebibliography}
\end{document}